\documentclass{amsart}

\usepackage{amsmath, amsfonts, mathrsfs}
\usepackage{amssymb,latexsym}
\usepackage{enumerate, bbm}
\usepackage{bbm} 
\usepackage[dvipsnames]{xcolor}

\usepackage[top=3.1cm, bottom=3.1cm, left=3.1cm, right=3.1cm, includefoot, heightrounded]{geometry}

\newtheorem{theorem}{Theorem}[section]

\newtheorem{lemma}[theorem]{Lemma}

\theoremstyle{definition}
\newtheorem{definition}[theorem]{Definition}

\numberwithin{equation}{section}

\begin{document}

\title[Measurable by seminorm selector]
{Multifunctions  admitting a measurable by seminorm selector in Frechet spaces}

\author{Sokol Bush Kaliaj }

\address{University of Elbasan,
Mathematics Department,
Elbasan, Albania}

\email{sokol\_bush@yahoo.co.uk}

\thanks{}

\subjclass[2010]{28B05, 28B20, 46B22, 46A04}

\keywords{Multifunctions,  Fr\'{e}chet spaces, measurable by seminorm selector.}

\begin{abstract}
In this paper we present full characterizations of multifunctions  admitting 
a measurable by seminorm selector in Frechet spaces.
\end{abstract}

\maketitle

\section{ Introduction and Preliminaries }

An important aspect when dealing with multifunctions is the existence of selections with appropriate properties depending on the properties of multifunctions. 
In this direction, 
when speaking about measurability, 
the most used result is the classical Kuratowski-Ryll-Nardzewski Theorem 
which ensures the existence of strongly measurable selections for Effros measurable multifunctions 
in complete metrizable spaces that are separable, 
see  \cite{KUR}.  
Without separability assumption,  
in Banach spaces the existence of strongly measurable selectors 
was obtained in the papers 
\cite{CASC1} and \cite{CASC2} by B.~Cascales, V.~Kadets, and J.~Rodr\'{i}guez. 
In non-normable non-separable spaces, we refer to the papers 
\cite{HANS1}-\cite{HANS2}, \cite{HIM} and \cite{MAG}.
A striking and pretty useful selection theorem relevant for this paper is the following:
\begin{theorem}(Cascales-Kadets-Rodr\'{i}guez \cite[Theorem 2.5]{CASC2})
Let $\mathfrak{X}$ be a Banach space 
and let $F: \Omega \to wk(\mathfrak{X})$ be 
a multifunction. 
Then the following statements are equivalent:
\begin{itemize} 
\item[(i)]
$F$ admits a strongly measurable selector,
\item[(ii)]
$F$ satisfies property $(P)$,
\item[(iii)]
There exist a set of measure zero $\Omega_{0} \in \Sigma$, 
a separable subspace $Y \subset \mathfrak{X}$ 
and a multifunction $G : \Omega \setminus \Omega_{0} \to wk(Y)$ 
that is Effros measurable and such that $G(t) \subset F(t)$ for every 
$t \in \Omega \setminus \Omega_{0}$. 
\end{itemize} 
\end{theorem}
The aim of this paper is to prove a version of 
this theorem in a Fr\'{e}chet space $X$, 
i.e. in a metrizable, 
complete, locally convex topological vector space $(X,\tau)$. 
It is well known that there exists a increasing sequence $(p_{i})_{i \in \mathbb{N}}$ 
of seminorms 
defining the topology $\tau$. 
The metric 
\begin{equation*}
\begin{split}
d(x,y) = \sum_{i=1}^{+\infty} \frac{1}{2^{i}} \frac{p_{i}(x-y)}{1+p_{i}(x-y)}
\end{split}
\end{equation*} 
is such that $\tau_{d} = \tau$, 
where  $\tau_{d}$ is the topology generated by $d$. 
For any $i \in \mathbb{N}$,  
$\tau_{i}$ is the topology of the semimetric space $(X, p_{i})$,  
$\widetilde{X}_{i}$ is the quotient vector space $X/p_{i}^{-1}(0)$, 
$\varphi_{i}: X \to \widetilde{X}_{i}$ is 
the canonical quotient map,  
$(\widetilde{X}_{i},\widetilde{p}_{i})$ is the quotient normed space 
($\widetilde{p}_{i}(\varphi_{i}(\cdot)) = p_{i}(\cdot)$) and 
$(\overline{X}_{i},\overline{p}_{i})$ is the completion of 
$(\widetilde{X}_{i},\widetilde{p}_{i})$.

Throughout this paper 
$(\Omega, \Sigma, \mu)$ is a complete finite measure space, 
$\Sigma^{+}$ is the family of all measurable sets $A \in \Sigma$ with $\mu(A) >0$ and 
$\Sigma^{+}_{A} = \{B \in \Sigma^{+}: B \subset A\}$, $A \in \Sigma^{+}$.
By $2^{X}$ we denote the family of all nonempty  subsets of $X$ 
and by $ck(X)$ we denote  
the subfamily 
of $2^{X}$ made up of convex compact subsets of $X$. 
It is clear that 
$ck(X) \subset wk(X)$, 
where 
$wk(X)$ is the subfamily 
of $2^{X}$ made up of weakly compact subsets of $X$.


\begin{definition}\label{D_pMeasurable}
A function 
$f : \Omega \to X$ is said to be \textit{measurable} if 
\begin{equation*}
\begin{split}
(\forall \mathcal{O} \in \tau)[f^{-1}(\mathcal{O}) \in \Sigma]. 
\end{split}
\end{equation*}
If $f$ has only a finite set 
$x_{1}, \dotsc, x_{n}$ of values 
then $f$ is said to be a \textit{simple function}; in this case,  
$f$  is measurable if and only if 
\begin{equation*}
\begin{split}
f^{-1}(x_{i}) = \{t \in \Omega : f(t) = x_{i}\} \in \Sigma, 
\quad
i = 1, \dotsc,n. 
\end{split}
\end{equation*}
\end{definition}

\begin{definition}\label{D_Blondia_M}
Let $f: \Omega \to X$ be a function. 
We say that $f$ is 
\begin{itemize}
\item[(i)]
\textit{$p_{i}$-strongly measurable}
if  there exist 
a sequence of measurable simple functions $(f_{n}^{i} : \Omega \to X)$ 
and a measurable set  
$Z_{i} \in \Sigma$ with $\mu(Z_{i})=0$ 
such that  
\begin{equation*}
\begin{split}
\lim_{n \to \infty} p_{i}(f_{n}^{i}(t) -f(t)) = 0
\quad\text{for all }t \in \Omega \setminus Z_{i},
\end{split}
\end{equation*}
\item[(ii)]
\textit{measurable by seminorm} 
if for each $i \in \mathbb{N}$ 
the function $f$ is $p_{i}$-strongly measurable, 
\item[(iii)]
\textit{strongly measurable} 
if there exist a sequence of measurable simple functions 
$(f_{n}: \Omega \to X)$ and a measurable set $Z \in \Sigma$ with $\mu(Z)=0$ such that
\begin{equation*}
\begin{split} 
\lim_{n \to \infty} f_{n}(t) = f(t) 
\quad\text{for all }t \in \Omega \setminus Z. 
\end{split}
\end{equation*} 
\end{itemize}
\end{definition}


Given a multifunction  $F: \Omega \to 2^{X}$ and a subset $U \subset X$, 
we write  
\begin{equation*}
\begin{split}
F^{-}(U) = \{ t \in \Omega : F(t) \cap U \neq \emptyset \}
\end{split}
\end{equation*}
and
\begin{equation*}
\begin{split}
\widetilde{F}_{i}: \Omega \to 2^{\widetilde{X}_{i}}, 
\quad 
\widetilde{F}_{i}(\omega) = \varphi_{i}( F(\omega) ),
\quad 
i \in \mathbb{N}.  
\end{split}
\end{equation*}
We say that $F: \Omega \to 2^{X}$ 
\textit{admits 
a measurable by seminorm selector} if there exists 
a measurable by seminorm function $f: \Omega \to X$ such that 
$f(t) \in F(t)$ for every $t \in \Omega$.

\begin{definition}\label{D_Effros_Measurable}
Let $F: \Omega \to 2^{X}$ be a multifunction. 
We say that 
\begin{itemize}
\item[(i)]
$F$ is \textit{$p_{i}$-Effros measurable} if 
\begin{equation*}
\begin{split}
(\forall \mathcal{O} \in \tau_{i})[F^{-}(\mathcal{O}) \in \Sigma],
\end{split}
\end{equation*}
\item[(ii)]
$\widetilde{F}_{i}$ is \textit{Effros measurable} if 
\begin{equation*}
\begin{split}
(\forall \widetilde{\mathcal{O}} \in \widetilde{\tau}_{i})[\widetilde{F}_{i}^{-}
(\widetilde{\mathcal{O}}) \in \Sigma],
\end{split}
\end{equation*}
where $\widetilde{\tau}_{i}$ be the topology of $(\widetilde{X}_{i}, \widetilde{p}_{i})$, 
\item[(iii)]
$F$ is \textit{Effros measurable} if 
\begin{equation*}
\begin{split}
(\forall \mathcal{O} \in \tau)[F^{-}(\mathcal{O}) \in \Sigma].
\end{split}
\end{equation*} 
\end{itemize}
\end{definition}


\begin{definition}\label{D_Property_P}
Let $F: \Omega \to 2^{X}$ be a multifunction. 
We say that
\begin{itemize}
\item[(i)]
\textit{$F$ satisfies property $\mathcal{P}_{i}$} $(i \in \mathbb{N})$, 
if for each $\varepsilon >0$  and each $A \in \Sigma^{+}$ 
there exist $B \in \Sigma^{+}_{A}$ and $D \subset X$ with 
$\text{diam}_{i}(D) \leq \varepsilon$ 
such that 
\begin{equation*}
B \subset F^{-}(D), 
\end{equation*}
where
$\text{diam}_{i}(D) = \sup \{p_{i}(x-y): x,y \in D\}$,
\item[(ii)]
\textit{$\widetilde{F}_{i}$ satisfies property $\mathcal{P}$} $(i \in \mathbb{N})$ 
if for each $\varepsilon >0$  and each $A \in \Sigma^{+}$ 
there exist $B \in \Sigma^{+}_{A}$ and $\widetilde{D} \subset \widetilde{X}_{i}$ with 
$\text{diam}_{i}(\widetilde{D}) \leq \varepsilon$ 
such that 
\begin{equation*}
B \subset \widetilde{F}_{i}^{-}(\widetilde{D}), 
\end{equation*}
where
$\text{diam}_{i}(\widetilde{D}) = \sup \{\widetilde{p}_{i}(\widetilde{x}-\widetilde{y}): 
\widetilde{x},\widetilde{y} \in \widetilde{D}\}$,
\item[(iii)]
\textit{$F$ satisfies property $\mathcal{P}$},  
if for each $\varepsilon >0$  and each $A \in \Sigma^{+}$ 
there exist $B \in \Sigma^{+}_{A}$ and $D \subset X$ with 
$\text{diam}(D) \leq \varepsilon$ 
such that 
$B \subset F^{-}(D)$, 
where
$\text{diam}(D) = \sup \{d(x,y): x,y \in D\}$.
\end{itemize} 
\end{definition}

\section{The Main Result}


The main result is Theorem \ref{T_Main}. 
Let us start with a few auxiliary lemmas.
The first lemma is proved in \cite[p.206]{KOTHE}.

\begin{lemma}\label{L_Kothe}
Let $(\alpha_{i})$ be an increasing sequence of positive
numbers and let 
$$
\alpha = \sum_{i=1}^{+\infty} \frac{1}{2^{i}} \frac{\alpha_{i}}{1+\alpha_{i}}.
$$
Then the following statements hold:
\begin{itemize}
\item[(i)]
If for some $k \in \mathbb{N}$,  $\alpha_{k} < \frac{1}{2^{k}}$, 
then $\alpha < \frac{1}{2^{k-1}}$.
\item[(ii)]
If for some $k, s \in \mathbb{N}$,  $\alpha < \frac{1}{2^{k}}\frac{1}{2^{s+1}}$, 
then $\alpha_{k} < \frac{1}{2^{s}}$.
\end{itemize}
\end{lemma}

The following lemma  
presents a characterization of the Effros measurability of 
a multifunction $F$ 
in terms of the Effros measurability of $\widetilde{F}_{i}$.


\begin{lemma}\label{Relationship_Effros}
Let $F: \Omega \to 2^{X}$ be a multifunction. 
Then the following statements are equivalent:
\begin{itemize}
\item[(i)]
$F$ is Effros measurable.
\item[(ii)]
$F$ is $p_{i}$-Effros measurable for every  $i \in \mathbb{N}$. 
\item[(iii)]
$\widetilde{F}_{i}$ is Effros measurable for every  $i \in \mathbb{N}$. 
\end{itemize}
\end{lemma}
\begin{proof}
$(i) \Rightarrow(ii)$ 
Assume that $F$ 
is Effros measurable, let $k \in \mathbb{N}$ and  
let $\mathcal{O} \in \tau_{k}$. 
Then, given $x \in \mathcal{O}$ there exists $s \in \mathbb{N}$ such that 
\begin{equation*}
B_{k}\left ( x, \frac{1}{2^{s}} \right ) \subset \mathcal{O}, 
\quad
B_{k}\left ( x, \frac{1}{2^{s}} \right ) = 
\left \{y \in X: p_{k}(x-y) < \frac{1}{2^{s}} \right \}.
\end{equation*} 
By Lemma \ref{L_Kothe} we have also
\begin{equation*}
B \left ( x, \frac{1}{2^{k}}\frac{1}{2^{s+1}} \right ) \subset B_{k}\left ( x, \frac{1}{2^{s}} \right ) \subset \mathcal{O}, 
\quad 
B \left ( x, \frac{1}{2^{k}}\frac{1}{2^{s+1}} \right ) 
= 
\left \{y \in X: d(x,y) < \frac{1}{2^{k}}\frac{1}{2^{s+1}} \right \}.
\end{equation*}
This means that $\mathcal{O} \in \tau_{d} = \tau$ 
and, consequently,  
$F^{-}(\mathcal{O}) \in \Sigma$. 
Thus, $F$ is $p_{k}$-Effros measurable.

$(ii) \Rightarrow (i)$  
Assume that $(ii)$ holds and let $\mathcal{O} \in \tau_{d} = \tau$. 
Then, for each $x \in \mathcal{O}$ there exists $k_{x} \in \mathbb{N}$ such that 
\begin{equation*}
B\left ( x, \frac{1}{2^{k_{x}-1}} \right ) \subset \mathcal{O}, 
\quad
B\left ( x, \frac{1}{2^{k_{x}-1}} \right ) = 
\left \{y \in X: d(x,y) < \frac{1}{2^{k_{x}-1}} \right \}
\end{equation*} 
and, by Lemma \ref{L_Kothe},
\begin{equation*}
B_{k_{x}} \left ( x, \frac{1}{2^{k_{x}}} \right ) \subset B\left ( x, \frac{1}{2^{k_{x}-1}} \right ) \subset \mathcal{O}, 
\quad 
B_{k_{x}} \left ( x, \frac{1}{2^{k_{x}}} \right )
= 
\left \{y \in X: p_{k_{x}}(x-y) < \frac{1}{2^{k_{x}}} \right \}.
\end{equation*}
Hence, for each $i \in \mathbb{N}$, we have
\begin{equation*} 
\mathcal{O}_{i} \in \tau_{i},\quad 
\mathcal{O}_{i} = \bigcup_{(x \in \mathcal{O}:k_{x} =i)} B_{k_{x}} \left ( x, \frac{1}{2^{k_{x}}} \right ), 
\quad
F^{-}(\mathcal{O}_{i}) \in \Sigma,
\end{equation*}
and since
\begin{equation*}
\mathcal{O} = \bigcup_{i=1}^{+\infty} \mathcal{O}_{i}
\quad
\text{and}
\quad
F^{-}(\mathcal{O}) = \bigcup_{i=1}^{+\infty} F^{-}(\mathcal{O}_{i})
\end{equation*}
it follows that $F^{-}(\mathcal{O}) \in \Sigma$. 
This means that $F$ is Effros measurable.

$(ii) \Leftrightarrow (iii)$ 
From the fact that $\varphi_{i}$ is an open map it follows that
\begin{equation*}
\begin{split}
\widetilde{\tau}_{i} = 
\{ \varphi_{i}(\mathcal{O}) : \mathcal{O} \in \tau_{i} \}.
\end{split}
\end{equation*}
Thus it is enough to prove that for any $\mathcal{O} \in \tau_{i}$ the following equality holds
\begin{equation}\label{eq_L_iEffros.1}
\begin{split}
F^{-}(\mathcal{O})= \widetilde{F}_{i}^{-}(\varphi_{i}(\mathcal{O})). 
\end{split}
\end{equation} 
It is clear that 
\begin{equation}\label{eq_L_iEffros.2}
\begin{split}
t \in F^{-}(\mathcal{O}) \Rightarrow F(t) \cap \mathcal{O} \neq \emptyset 
\Rightarrow 
\widetilde{F}_{i}(t) \cap \varphi_{i}(\mathcal{O}) \neq \emptyset 
\Rightarrow t \in \widetilde{F}_{i}^{-}(\varphi_{i}(\mathcal{O})).
\end{split}
\end{equation}
Conversely, 
\begin{equation*}
\begin{split}
t \in \widetilde{F}_{i}^{-}(\varphi_{i}(\mathcal{O})) 
\Rightarrow
\widetilde{F}_{i}(t) \cap \varphi_{i}(\mathcal{O}) \neq \emptyset
\Rightarrow&
(\exists x \in X)[\varphi_{i}(x) \in \widetilde{F}_{i}(t)\text{ and }
\varphi_{i}(x) \in \varphi_{i}(\mathcal{O})] \\
\Rightarrow& 
(\exists y \in F(t)) 
(\exists z \in \mathcal{O}) 
[\varphi_{i}(x) = \varphi_{i}(y) = \varphi_{i}(z)].
\end{split}
\end{equation*}
Since $z \in \mathcal{O}$ there exists $\varepsilon >0$ such that
\begin{equation*}
\begin{split}
B_{i}(z, \varepsilon) \subset \mathcal{O}, 
\quad
B_{i}(z, \varepsilon) = \{w \in X: p_{i}(z-w) < \varepsilon\},
\end{split}
\end{equation*}
and since $p_{i}(z-y) = 0$ it follows that 
$y \in B_{i}(z, \varepsilon) \subset \mathcal{O}$ and $y \in F(t)$. 
Thus,
\begin{equation*}
\begin{split}
t \in \widetilde{F}_{i}^{-}(\varphi_{i}(\mathcal{O})) 
\Rightarrow 
\widetilde{F}_{i}(t) \cap \varphi_{i}(\mathcal{O}) \neq \emptyset
\Rightarrow
F(t) \cap \mathcal{O} \neq \emptyset 
\Rightarrow 
t \in F^{-}(\mathcal{O}). 
\end{split}
\end{equation*}
The last result together with \eqref{eq_L_iEffros.2} yields that 
\begin{equation*}
\begin{split}
t \in F^{-}(\mathcal{O}) 
\Leftrightarrow t \in \widetilde{F}_{i}^{-}(\varphi_{i}(\mathcal{O})). 
\end{split}
\end{equation*}
This means that \eqref{eq_L_iEffros.1} holds and the proof is finished.
\end{proof}


The next lemma presents a characterization of the property  $\mathcal{P}$ 
for $F$ in terms of the property  $\mathcal{P}$ for $\widetilde{F}_{i}$. 

\begin{lemma}\label{Relationship_P}
Let $F: \Omega \to 2^{X}$ be a multifunction. 
Then the following statements are equivalent:
\begin{itemize}
\item[(i)]
$F$ satisfies property  $\mathcal{P}$. 
\item[(ii)]
$F$ satisfies property  $\mathcal{P}_{i}$ for every  $i \in \mathbb{N}$. 
\item[(iii)]
$\widetilde{F}_{i}$ satisfies property  $\mathcal{P}$ for every  $i \in \mathbb{N}$.
\end{itemize}
\end{lemma}
\begin{proof}
$(i) \Rightarrow(ii)$
Assume that $F$ satisfies property  $\mathcal{P}$ 
and let $\varepsilon >0$, $A \in \Sigma^{+}$ and $k \in \mathbb{N}$. 
Then, there exists $s \in \mathbb{N}$ such that  
$\frac{1}{2^{s}} < \varepsilon$. 
Hence, there exist $B \in \Sigma^{+}_{A}$ and $D \subset X$ with 
$\text{diam}(D) < \frac{1}{2^{k}}\frac{1}{2^{s+1}}$ 
such that 
\begin{equation*} 
F(t) \cap D \neq \emptyset \quad\text{for all }t \in B.
\end{equation*} 
Thanks to Lemma \ref{L_Kothe} it follows that 
$\text{diam}_{k}(D) \leq \frac{1}{2^{s}} < \varepsilon$.
Thus, $F$ satisfies property  $\mathcal{P}_{k}$.

$(ii) \Rightarrow(i)$ 
Assume that $(ii)$ holds and 
let  $\varepsilon >0$  and $A \in \Sigma^{+}$.  
There exists $k \in \mathbb{N}$ such that 
$\frac{1}{2^{k-1}} < \varepsilon$, 
and since $F$ satisfies property  $\mathcal{P}_{k}$ 
there exist $B \in \Sigma^{+}_{A}$ and $D \subset X$ with 
$\text{diam}_{k}(D) < \frac{1}{2^{k}}$ 
such that 
\begin{equation*}
F(t) \cap D \neq \emptyset \quad\text{for all }t \in B. 
\end{equation*}
The last inequality together with Lemma \ref{L_Kothe} yields that 
$\text{diam}(D) \leq \frac{1}{2^{k-1}} <  \varepsilon$. 
This means that $F$ satisfies property  $\mathcal{P}$.

By Definition \ref{D_Property_P} it follows that 
$(ii)$ and $(iii)$ are equivalent, and this ends the proof.
\end{proof}


\begin{lemma}\label{L_KeyLemma}
Let $(C_{i})$ be a sequence of nonempty compact  subsets $C_{i} \subset X$ 
with respect to $p_{i}$ 
such that 
\begin{equation*}
\begin{split}
C_{i} \supset C_{i+1}
\quad
\text{for every }i \in \mathbb{N}.
\end{split}
\end{equation*}
Then
\begin{equation*}
\begin{split}
\bigcap_{i \in \mathbb{N}} C_{i} \neq \emptyset.
\end{split}
\end{equation*}
\end{lemma}
\begin{proof}
Since each $C_{i}$ is nonempty we can choose $x_{i} \in C_{i}$ for every $i \in \mathbb{N}$. 
Then $(x_{i})$ is a sequence with terms in $C_{1}$. 
From the fact that $C_{1}$ is a compact subset in $(X, p_{1})$ it follows that there exists 
a subsequence $(x_{k_{1}(i)})$ of $(x_{i})$ and $y_{1} \in C_{1}$ such that 
\begin{equation*}
\begin{split}
\lim_{i \to \infty } p_{1}(x_{k_{1}(i)} - y_{1}) =0.
\end{split}
\end{equation*}
We have that $(x_{k_{1}(i)})_{i>1}$ is a sequence with terms in $C_{2}$. 
From the fact that $C_{2}$ is a compact subset in $(X, p_{2})$ it follows that there exists 
a subsequence $(x_{k_{2}(i)})$ of $(x_{k_{1}(i)})_{i>1}$ and $y_{2} \in C_{2}$ such that 
\begin{equation*}
\begin{split}
\lim_{i \to \infty } p_{2}(x_{k_{2}(i)} - y_{2}) =0. 
\end{split}
\end{equation*}
The sequence $(x_{k_{2}(i)})_{i>2}$ is  with terms in $C_{3}$. 
Again, from the fact that $C_{3}$ is a compact subset in $(X, p_{3})$ 
it follows that there exists 
a subsequence $(x_{k_{3}(i)})$ of $(x_{k_{2}(i)})_{i>2}$ 
and $y_{3} \in C_{3}$ such that 
\begin{equation*}
\begin{split}
\lim_{i \to \infty } p_{3}(x_{k_{3}(i)} - y_{3}) =0.
\end{split}
\end{equation*}
By proceeding inductively in this fashion we obtain a sequence  
$(y_{s})$ with $y_{s} \in C_{s}$ and  
a sequence $(x_{k_{s}(i)})$ with terms in $C_{s}$ 
such that 
\begin{equation*}
\begin{split}
\lim_{i \to \infty } p_{s}(x_{k_{s}(i)} - y_{s}) =0
\quad
\text{for every }s \in \mathbb{N}. 
\end{split}
\end{equation*}
Since
\begin{equation*}
\begin{split}
s' \leq s'' 
\Rightarrow& 
\lim_{i \to \infty } p_{s'}(x_{k_{s''}(i)} - y_{s'}) =0 
\text{ and }
\lim_{i \to \infty } p_{s'}(x_{k_{s''}(i)} - y_{s''}) =0 
\\
\Rightarrow&
p_{s'}(y_{s'}-y_{s''}) =0
\end{split}
\end{equation*}
it follows that $(y_{s})$ 
is a Cauchy sequence in the complete metric space $(X,\tau)$. 
Consequently, $(y_{s})$ converges to 
a point $y_{0}$ in $(X,\tau)$. 
Since $y_{0}$ is a limit point of $(x_{i})$ we have 
\begin{equation*}
\begin{split}
y_{0} \in 
\bigcap_{i \in \mathbb{N}} \overline{A_{i}}^{\tau} 
\subset 
\bigcap_{i \in \mathbb{N}} \overline{A_{i}}^{p_{i}}
\subset 
\bigcap_{i \in \mathbb{N}} C_{i},
\end{split}
\end{equation*}
where 
$A_{i} = \{x_{j}: j \geq i \}$, 
and this ends the proof.
\end{proof}


\begin{lemma}\label{L_KKeyLemma}
Let $(C_{i})$ be a sequence of nonempty compact  subsets 
$C_{i} \subset X$ with respect to $p_{i}$ 
such that 
\begin{equation*}
\begin{split}
C_{i} \supset C_{i+1}
\quad
\text{for every }i \in \mathbb{N}
\end{split}
\end{equation*}
and let $C =\bigcap_{i \in \mathbb{N}} C_{i}$. 
Then for each $k \in \mathbb{N}$ and $x \in X$ the following equality holds 
\begin{equation*}
\begin{split}
\text{dist}_{k}(x, C) = \inf_{i \geq k} \text{dist}_{k}(x, C_{i}), 
\end{split}
\end{equation*}
where 
$\text{dist}_{k}(x, A) = \inf \{p_{k}(x-y): y \in A\}, A \subset X$. 
\end{lemma}
\begin{proof}
Fix an integer $k \in \mathbb{N}$ and write
\begin{equation*}
\begin{split}
u = \text{dist}_{k}(x, C), \quad 
v =  \inf_{i \geq k} \text{dist}_{k}(x, C_{i}).
\end{split}
\end{equation*}
Since $C =\bigcap_{i \geq k} C_{i}$ we obtain $u \geq v$. 
On the other hand, 
given $\varepsilon >0$  
there exists a sequence  $(x_{s})$ with $x_{s} \in C_{k+s}$ such that 
\begin{equation}\label{eqL_KKeyLemma.1}
\begin{split}
v \leq p_{k}(x - x_{s}) < v + \frac{\varepsilon}{2^{s}}, 
\qquad
s = 0, 1,2,3, \dotsc
\end{split}
\end{equation}
From the fact that $(x_{s})$ is a sequence with terms in the compact set $C_{k}$ 
it follows that there is a subsequence $(x_{s_{j}})$ and $x_{0} \in C_{k}$ such that 
\begin{equation}\label{eqL_KKeyLemma.2}
\begin{split}
\lim_{j \to \infty} p_{k}(x_{s_{j}}- x_{0}) =0.
\end{split}
\end{equation}
Consequently,
\begin{equation*}
\begin{split}
x_{0} \in \bigcap_{j \in \mathbb{N}} C_{k+s_{j}} = C
\end{split}
\end{equation*}
and by 
\eqref{eqL_KKeyLemma.1} and \eqref{eqL_KKeyLemma.2} 
we obtain also $p_{k}(x - x_{0}) \leq v$, 
since
\begin{equation*}
\begin{split}
p_{k}(x - x_{0}) \leq p_{k}(x - x_{s_{j}}) + p_{k}(x_{s_{j}}- x_{0}). 
\end{split}
\end{equation*}
It follows that $u \leq v$, 
and since $u \geq v$ we infer that $u=v$, 
and this ends the proof. 
\end{proof}


We are now ready to present the main result.

\begin{theorem}\label{T_Main}
Let $F: \Omega \to ck(X)$ be 
a multifunction. 
Then the following statements are equivalent:
\begin{itemize}
\item[(i)]
$F$ admits a measurable by seminorm selector,
\item[(ii)]
$F$ satisfies the property $\mathcal{P}$,
\item[(iii)]
There exist a set of measure zero $\Omega_{0} \in \Sigma$, 
a separable subspace $Y \subset X$ 
and a multifunction $G : \Omega \setminus \Omega_{0} \to ck(Y)$ 
that is Effros measurable and such that $G(t) \subset F(t)$ for every 
$t \in \Omega \setminus \Omega_{0}$. 
\end{itemize}  
\end{theorem}
\begin{proof}
$(i) \Rightarrow (ii)$  
Assume that $F$ admits a measurable by seminorm selector $f: \Omega \to X$. 
Then the multifunction 
$\widetilde{F}_{i} : \Omega \to ck(\widetilde{X}_{i})$ 
admits a strongly measurable selector $\varphi_{i} \circ f$. 
Since $ck(\widetilde{X}_{i}) \subset wk(\widetilde{X}_{i})$,   
by \cite[Theorem 2.5]{CASC2} 
each multifunction $\widetilde{F}_{i}$ satisfies the property $\mathcal{P}$.
Further by Lemma \ref{Relationship_P} it follows that $F$ 
satisfies the property $\mathcal{P}$.

$(ii) \Rightarrow (iii)$ 
Assume that $F$  
satisfies the property $\mathcal{P}$. 
Then by Lemma \ref{Relationship_P} each multifunction 
$\widetilde{F}_{i}$ satisfies the property $\mathcal{P}$. 
Therefore by \cite[Theorem 2.5]{CASC2}  
there exist a measurable set $\Omega_{i} \in \Sigma$ with $\mu(\Omega_{i})=0$, 
a separable subspace $\widetilde{Y}_{i} \subset \widetilde{X}_{i}$ 
and a multifunction 
$\widetilde{G}_{i} : \Omega \setminus \Omega_{i} \to wk(\widetilde{Y}_{i})$ 
that is Effros measurable and such that 
$\widetilde{G}_{i}(t) \subset \widetilde{F}_{i}(t)$ for every 
$t \in \Omega \setminus \Omega_{i}$. 

From the fact that $\varphi_{i}$ is an open map it follows that 
\begin{equation*}
\begin{split}
F_{i}(t) =\varphi_{i}^{-1}[\widetilde{F}_{i}(t)] 
= \{x \in X : \varphi_{i}(x) \in \widetilde{F}_{i}(t)\},
\qquad
t \in \Omega
\end{split}
\end{equation*}
is a compact subset of the semimetric space $(X,p_{i})$. 
Since $\varphi_{i}$ is also an isometric map the set
\begin{equation*}
\begin{split}
Y_{i} = \varphi_{i}^{-1}(\widetilde{Y}_{i}) 
= \{x \in X : \varphi_{i}(x) \in \widetilde{Y}_{i} \}
\end{split}
\end{equation*}
is a separable subset of $(X,p_{i})$.

We can define the multifunctions 
\begin{equation*}
\begin{split}
G_{i} : \Omega \setminus \Omega_{i} \to 2^{Y_{i}}, 
\quad
G_{i}(t) =\varphi_{i}^{-1}[\widetilde{G}_{i}(t)] 
= \{x \in Y_{i} : \varphi_{i}(x) \in \widetilde{G}_{i}(t)\},
\end{split}
\end{equation*}
and
\begin{equation}\label{eq_T_Main.1}
\begin{split}
G : \Omega \setminus \Omega_{0} \to 2^{Y},
\quad
G(t) = \bigcap_{i=1}^{+\infty} \overline{H_{i}(t)}^{p_{i}},
\quad
H_{i}(t) = \bigcup_{j \geq i} G_{j}(t),
\end{split}
\end{equation}
where  
\begin{equation*}
\begin{split}
\Omega_{0} = \bigcup_{i=1}^{+\infty} \Omega_{i}, 
\quad
Y = \overline{\text{sp}}^{\tau}
\left ( 
U
\right ),
\quad
U = \bigcup_{i=1}^{+\infty}  \overline{Y_{i}}^{p_{i}},
\end{split}
\end{equation*}
and $\overline{\text{sp}}^{\tau}(U)$ 
is the closed subspace spanned by $U$.  
It is easy to see that $\mu(\Omega_{0})=0$ and $Y$ is a separable subspace of $(X, \tau)$.

$\bullet$ 
We claim that $G(t) \neq \emptyset$ for any $t \in \Omega \setminus \Omega_{0}$. 
Indeed, 
since each $F_{i}(t)$ is a compact subset of $(X,p_{i})$ 
and 
\begin{equation*}
\begin{split}
F_{i}(t) \supset F_{i+1}(t)
\end{split}
\end{equation*}
it follows that for each $i \in \mathbb{N}$, we have 
\begin{equation*}
\begin{split}
\overline{H_{i}(t)}^{p_{i}}  \subset \overline{F_{i}(t)}^{p_{i}} = F_{i}(t).
\end{split}
\end{equation*}
Thus 
each $\overline{H_{i}(t)}^{p_{i}}$ is a nonempty compact subset of $(X,p_{i})$. 
Consequently, 
we obtain by \eqref{eq_T_Main.1} and Lemma \ref{L_KeyLemma} 
that  $G(t) \neq \emptyset$.

$\bullet$ 
For each $t \in \Omega$ we will prove 
\begin{equation*}
\begin{split} 
\bigcap_{i \in \mathbb{N}} F_{i}(t) = F(t),
\end{split}
\end{equation*}
Form the fact that $F_{i}(t) \supset F(t)$ 
it follows that 
\begin{equation*}
\begin{split} 
\bigcap_{i \in \mathbb{N}} F_{i}(t) \supset F(t).
\end{split}
\end{equation*}
To see the converse inclusion we consider 
$x \not\in F(t)$. 
Then, since $F(t)$ is a closed set in $(X, \tau)$ 
there is $i_{0} \in \mathbb{N}$ and $\varepsilon_{0}>0$ such that 
\begin{equation*}
\begin{split}
B_{p_{i_{0}}}(x, \varepsilon_{0}) \bigcap F(t) = \emptyset, 
\quad
B_{p_{i_{0}}}(x, \varepsilon_{0}) 
= \{ y \in X: p_{i_{0}}(x - y) < \varepsilon_{0}\}.
\end{split}
\end{equation*}
Consequently, 
$x \not \in \bigcap_{i \in \mathbb{N}} F_{i}(t)$, 
since $x \not \in F_{i_{0}}(t)$.

$\bullet$ 
Given any $t \in \Omega \setminus \Omega_{0}$ we  claim that 
\begin{equation*}
\begin{split}
G(t) \subset F(t). 
\end{split}
\end{equation*}
Indeed, since
\begin{equation*}
\begin{split}
H_{i}(t) \subset F_{i}(t)
\end{split}
\end{equation*}  
it follows that 
\begin{equation*}
\begin{split}
G(t) = 
\bigcap_{i=1}^{+\infty} \overline{H_{i}(t)}^{p_{i}}
\subset 
\bigcap_{i=1}^{+\infty} \overline{F_{i}(t)}^{p_{i}} = 
\bigcap_{i=1}^{+\infty} F_{i}(t) = F(t).
\end{split}
\end{equation*}

$\bullet$ 
Given any $k \in \mathbb{N}$ we claim that $G$ is $p_{k}$-Effros measurable. 
Indeed, 
for each $i \in \mathbb{N}$ the multifunction
\begin{equation*}
\begin{split}
H_{i} : \Omega \setminus \Omega_{0} \to 2^{Y},
\quad
H_{i}(t) = \bigcup_{j \geq i} G_{j}(t)
\end{split}
\end{equation*}
is $p_{i}$-Effros measurable, 
since $\tau_{i} \subset \tau_{j}$ for all $j \geq i$, 
each $G_{j}$ is $p_{j}$-Effros measurable and
\begin{equation*}
\begin{split}
H_{i}^{-}(\mathcal{O}) = \bigcup_{j \geq i} G_{j}^{-}(\mathcal{O}),
\quad
\text{for every }\mathcal{O} \in \tau_{i}.
\end{split}
\end{equation*}
Consequently, the multifunction
\begin{equation*}
\begin{split}
T_{i}: \Omega \setminus \Omega_{0} \to 2^{Y}, 
\quad
T_{i}(t) = \overline{H_{i}(t)}^{p_{i}}
\end{split}
\end{equation*}
is also $p_{i}$-Effros measurable.

Note that
\begin{equation*}
\begin{split}
G(t) = 
\bigcap_{i \geq k} T_{i}(t)
\end{split}
\end{equation*}
and $T_{i}$ is $p_{k}$-Effros measurable for each $i \geq k$, 
since $\tau_{k} \subset \tau_{i}$ for every  $i \geq k$.
Hence, for each $y \in Y$ and $i \geq k$, 
by \cite[Proposition 1.4, p.142]{HU} the function 
\begin{equation*}
\begin{split}
h_{y}^{i} : \Omega \setminus \Omega_{0} \to Y, 
\quad
h_{y}^{i}(t) = \text{dist}_{k}(y, T_{i}(t))
\end{split}
\end{equation*}
is a real valued measurable function, 
and since by Lemma \ref{L_KKeyLemma} the function
\begin{equation*}
\begin{split}
h_{y} : \Omega \setminus \Omega_{0} \to Y, 
\quad
h_{y}(t) = \text{dist}_{k}(y, G(t)) 
\end{split}
\end{equation*}
is such that
\begin{equation*}
\begin{split}
h_{y}(t) = \inf_{i \geq k} h_{y}^{i}(t)
\end{split}
\end{equation*}
it follows that $h_{y}$ is also measurable. 
Again, by \cite[Proposition 1.4, p.142]{HU} the multifunction $G$  
is $p_{k}$-Effros measurable.

$\bullet$ Since $G$  
is $p_{k}$-Effros measurable for every $k \in \mathbb{N}$ 
we obtain by Lemma \ref{Relationship_Effros} that $G$ is Effros measurable.

$\bullet$ We can assume that $G(t) \in ck(Y)$ for all $t \in \Omega \setminus \Omega_{0}$, 
since $G(t) \subset F(t)$ and 
$F(t) \in ck(X)$ for every  $t \in \Omega \setminus \Omega_{0}$.

$(iii) \Rightarrow (i)$ 
Assume that $(iii)$ holds. 
Then by Kuratowski-Ryll Nardzeiski's theorem 
there exists a measurable selector 
$g: \Omega \setminus \Omega_{0} \to Y$ of $G$. 
Define a function $f:\Omega \to X$ as $f(t) =g(t)$ for every $t \in \Omega \setminus \Omega_{0}$ 
and $f(t)$ as any point  of $F(t)$ for $t \in \Omega_{0}$.  
Then $f$ is a selector of $F$ that is  $p_{i}$-strongly measurable 
for every $i \in \mathbb{N}$, 
by \cite[Theorem II.1.2, p.42]{DIES}. 
Consequently, the function $f$ is measurable by seminorm. 
Thus the multifunction $F$ 
admits a measurable by seminorm selector, 
and this ends the proof. 
\end{proof}


\begin{thebibliography}{plain}







\bibitem{CASC1} B.~Cascales, V.~Kadets, and J.~Rodr\'{i}guez,  
\textit{Measurable selectors and set-valued Pettis integral in non-separable Banach spaces}, 
J. Funct. Anal. \textbf{256/3} (2009), 673-699.




\bibitem{CASC2} B.~Cascales, V.~Kadets, and J.~Rodr\'{i}guez, 
\textit{Measurability and selections of multifunctions in Banach spaces}, 
J. Convex Analysis \textbf{17} (2010), 229-240. 





\bibitem{CAST} C. ~Castaing, ~M. ~Valadier, 
\textit{Convex Analysis and Measurable Multifunctions} ,
Lect. Notes Math., \textbf{580}, Springer-Verlag, Berlin (1977).





\bibitem{DIES} J.~Diestel and J.~J.~Uhl, 
\textit{Vector Measures}, 
Math. Surveys, vol.15, Amer. Math. Soc., Providence, RI, 1977.






\bibitem{KOTHE} G.~K\"{o}the, 
\textit{Topological Vector Spaces I}, 
Springer-Verlag, (1983).







\bibitem{HANS1} R.~W.~Hansell, 
\textit{Extended Bochner measurable selectors}, 
Math. Ann., \textbf{277} (1987), 79-94.






\bibitem{HANS2} R.~W.~Hansell, 
\textit{Hereditarily-additive families in descriptive set theory and Borel measurable
multimaps}, 
Trans. AMS, \textbf{278(2)} (1983), 725-749.





\bibitem{HIM} C.J.Himmelberg, F.S. Van Vleck K.Prikry, 
\textit{The Hausdorff metric and measurable selections}, 
Topology and its Applications, \textbf{20} (1985), 121-133.



\bibitem{HU} S. ~Hu, N.~S.~Papageorgiou, 
\textit{Handbook of Multivalued Analysis, Vol.I}, 
Kluwer Academic Publishers, (1997).  


\bibitem{KUR} K.~Kuratowski and C.~Ryll-Nardzewski, 
\textit{A general theorem on selectors}, 
Bull. Acad. Polon. Sci. S\'{e}r. Sci. Math. Astronom. Phys. 13 (1965), 397-403.





\bibitem{MAG} G. M\"{a}gerl,
\textit{A unified approach to measurable and continuous selections}, 
Trans. AMS, \textbf{245} (1978), 443-452.








\end{thebibliography}
\end{document}